\documentclass[11pt]{article}

\usepackage{amscd,amsmath,amssymb,amsfonts,amsthm,latexsym,mathrsfs}
\usepackage[utf8]{inputenc}
\usepackage[T1]{fontenc}
\usepackage{enumitem}
\usepackage{bbold}

\usepackage[all]{xy}
\usepackage{graphicx}
\usepackage{tikz}
\usepackage{hyperref}

\usepackage{subcaption}
\usepackage{graphpap}
\usepackage{makeidx,showidx}
\usepackage{wasysym} 
\usepackage{xcolor}
\usepackage{multicol}
\usepackage{marvosym}
\usepackage{lipsum}
\usepackage{mathtools}

\newtheorem{thm}{Theorem}[section]
\newtheorem{cor}[thm]{Corollary}

\theoremstyle{definition}

\newtheorem{lem}[thm]{Lemma}
\newtheorem{prop}[thm]{Proposition}
\theoremstyle{definition}
\newtheorem{defn}[thm]{Definition}
\newtheorem{rem}[thm]{Remark}

\numberwithin{equation}{section}

\definecolor{verdeazul}{rgb}{0.30,0.80,0.30}
\definecolor{miazul}{rgb}{0.60,0.75,0.90}

\def\P{\mathbb{CP}} 
\def\C{\mathbb{C}} 
\def\Z{\mathbb{Z}} 
\def\O{\mathcal{O}} 
\def\F{\mathcal{F}} 
\def\R{\mathbb{R}} 
\title{Analysis of the weight Diagram Associated with Foliations on the $\P^{2}$.}
\author{P. Rubí Pantaleón-Mondragón}

\begin{document}
\maketitle

\begin{abstract}
   We analyze the weight diagram associated with foliations on the complex projective plane through the Hilbert–Mumford criterion in Geometric Invariant Theory, focusing in particular on invariants such as the algebraic multiplicity and the existence of invariant curves.
\end{abstract}

\section{Introduction}
A foliation $\F_{d}$ of degree $d\geq 2$ with isolated singularities on the complex projective plane $\P^{2}$ (with  homogeneous coordinate $x,y,z$) is determined by a vector field \[X(\F_{d})=X=P(x,y,z)\frac{\partial}{\partial x}+Q(x,y,z)\frac{\partial}{\partial y}+R(x,y,z)\frac{\partial}{\partial z},\]
where $P,Q$ and $R$ are homogeneous polynomials of degree $d$ without common factor, defined up to  the addition of a multiple of the radial vector field by a  homogeneous polynomial of degree $d-1$. We will use the terms foliation $\F_{d}$ and vector field $X$ interchangeably whenever no confusion arises.

A point $p\in\P^{2}$ such that $X(p)=\alpha p$ for some $\alpha\in\C$, is called a singular point of the foliation $\F_{d}$. According to Gómez-Mont, Kempf, Campillo, and Olivares (\cite{GM-K,CO}), foliations  can be studied although their singularities. That is, one can analyze the type of singularity, its algebraic multiplicity, Milnor number, and the existence (or not) of invariant curves, among other properties, with the aim of classifying them. 

A powerful tool for  classifying  algebraic objects, which can also be applied to the study of foliations, is Geometric Invariant Theory (GIT). This theory decomposes a variety into strata of two types:  semistable points (which include the stable points) and unstable points, according to the action of a reductive group.  The usual action of a group $G$ on a variety $V$ identifies two objects as belonging to the same orbit if and only if they are isomorphic.  Gómez-Mont and Kempf \cite{GM-K} proved that, in the space of foliations, if  all the singularities of a foliation are distinct, then  the foliation belongs to the stable stratum, which is a dense open subset of the foliation space.  This  condition is not necessary: there exist stable foliations with repeated singularities  \cite{castorena2024git}. However, it is not trivial to provide explicit examples in general. Moreover, there is no explicit formula describing the foliation lying in each stratum with a particular property. For certain special cases, some results provide explicit examples of foliations, although they do not fully classify their singularities \cite{CR16,AlcantaraGrado1,AlcantaraGrado2}. 

When the foliation has a unique singular point, more results are known; see, for example \cite{AP19,L22,CDGM,castorena2024git,pantaleon16problema}. In \cite{A11}, the author describes the type of foliations of degree $2$ appearing in each stratum according to the  existence of invariant curves. Moreover, in \cite{A10U},  Alcántara classifies unstable foliations of degree $d\geq 2$ depending on their automorphism group. 

Another important interest in studying the stability of foliations lies in its relation to providing counterexamples to the Poincaré problem \cite{poincare1891integration}. This problem concerns the conditions on the degree of a first integral in terms of the information of the foliation. In \cite{alcantara2024some}, the authors present families of foliations with a unique singular point associated to pencils whose degree increases. That is, the degree of the first integral associated to a foliation is not bounded by the degree of the defining 1-form.  Although they mention that there are no concrete results to determine the relation between the instability of foliations and the instability of their associated pencil; the family they present consists of unstable foliations.

The preceding results motivate analysis of  the weight diagram associated with the space of foliations, in order to derive additional properties of a foliation from its stability. Thus, the organization of the  paper is the following: in the section \ref{preliminares} recalling some standard material on the GIT, the Mumford function, foliation and some of their invariant. In the section \ref{results} we present our results, we will analyze the weight diagram associated with foliations analyzing its multiplicity and existence of invariant curves.

\section{Preliminaries}\label{preliminares}

\subsection{Stability of Mumford}
Let $V\subset \P^{n}$ be a non-singular complex projective variety, and let $G$ be a reductive group acting linearly on $V$.

Let $\lambda:\C^{*}\rightarrow G$ be a one-parameter subgroup (1-PS) of $G$. There is a morphism, which we denote again by $\lambda$, given by

\[\begin{array}{rcl}
  \lambda:\C^{*} &\rightarrow & GL_{n+1}(\C) \\
    t &\mapsto & \begin{array}{ccc}
      \lambda(t):  \C^{n+1} &\rightarrow & \C^{n+1}  \\
         \nu & \mapsto & \lambda(t)\nu.
    \end{array}
  \end{array}\]

 This morphism is a diagonal representation of $\C^{*}$. Hence, there exists a basis $\{\nu_{0},\ldots,\nu_{n}\}$ of  $\C^{n+1}$ such that, for every $t\in\C^{*}$, \[\lambda(t)\nu_{i}=t^{r_{i}}\nu_{i}, \mbox{ with } r_{i}\in \Z.\] The number $r_{i}$  is called the {\bf weight} of $\nu_{i}$ with respect to the action $\lambda$ on $\C^{n+1}$.

With the above  notation.

\begin{defn} Let $x\in V\subset \P^{n}$ be a point,  and let $\lambda$ be a $1$-PS of $G$. If $\overline{x}\in\C^{n+1}$ is a representative of the class $x$, and  $\overline{x}=\sum_{i=0}^{n}a_{i}\nu_{i}$, then $\lambda(t)\overline{x}=\sum_{i=0}^{n} t^{r_{i}}a_{i}\nu_{i}$. The {\bf Mumford function} is defined as  
\[\mu(x,\lambda):=min\{r_{i}:a_{i}\neq 0\}.\]
\end{defn}

The Hilbert-Mumford criterion, used to detect unstable points under a linearly action,  can be stated as the following numerical condition:

\begin{thm}[Definition][Theorem 4.9, \cite{newstead51lectures}] \label{definicionEstabilidad}
\begin{enumerate}
    \item $x$ is stable (or semistable) if and only if $\mu(x,\lambda)<0$ $($  or $\mu(x,\lambda)\leq 0)$ for every $1$-PS $\lambda$ of $G$.
    \item $x$ is unstable if and only if there exists a $1$-PS $\lambda$ of $G$ such that $\mu(x,\lambda)>0$.
\end{enumerate}
\end{thm}

If we consider the open subset of semistable points $V^{ss}$ and we restrict the action of $G$ on $V^{ss}$, then there exists a good quotient of this action. Moreover, Kirwan shows that the unstable points contains information about the quotient variety \cite{Kirwan}.

\begin{defn}
 Let $x\in V$. The automorphism group of $x$ is the stabilizer in $G$ of $x$.   
\end{defn}

\subsection{Foliations}
A degree $d$ foliation $\F_{d}$ on $\P^{2}$ (up to non-zero scalar multiplication) is defined by a vector field  
\begin{align}\label{Foliacion}
X=P(x,y,z)\frac{\partial}{\partial
       x}+Q(x,y,z)\frac{\partial}{\partial
       y}+R(x,y,z)\frac{\partial}{\partial z},
       \end{align}
where $P(x,y,z),Q(x,y,z)$ and $R(x,y,z)\in\C[x,y,z]$ are homogeneous polynomials of degree
$d$. Moreover,  for any homogeneous polynomial $G(x,y,z)\in \C[x,y,z]$ of
degree $d-1$, the vector fields \[X \mbox{ and }
X+G(x,y,z)\cdot(x\frac{\partial }{\partial
  x}+y\frac{\partial}{\partial y}+z\frac{\partial}{\partial z})\] define 
the same foliation. In this case, we say that they are equivalent.

We denote  by $\F(d;2)$ the space of foliations of degree $d$ on $\P^{2}$. According to \cite{GO}, the set $\F(d;2)$ is a projective space of dimension $d^{2}+4d+2$. 

We denote by $Sing(\F_{d})$ (or $Sing(X)$) the set of singular points of the foliation $\F_{d}$, that is, the points $p\in\P^{2}$ such that \[X(p)=\alpha p \mbox{ for some } \alpha\in \C.\] 

If $p=[1:0:0]\in\P^{2}$ is an isolated singularity  of the foliation $\F_{d}$. Then one can associate certain numerical invariants to $p$. 

Consider the vector field
\[  \chi_{d}=-g(y,z)\frac{\partial}{\partial y}+f(y,z)\frac{\partial}{\partial z},\] which gives a local representation of $\F_{d}$ in an open chart of $\P^{2}$ containing $p$.
The {\bf Milnor number} of $\F_{d}$ at $p$ is defined by \[\mu_{p}(\F_{d}):=\dim_{\C}\frac{\O_{0}(\C^{2})}{\langle f,g\rangle \cdot \O_{0}(\C^{2})},\] where $\O_{0}(\C^{2})$  denote the ring of regular functions at $(0,0)$.

By Jouanolou (\cite{J06}), \[\mu(\F_{d})=\sum\limits_{p\in Sing(\F_{d})} \mu_{p}(\F_{d})=d^{2}+d+1.\]

Moreover, in this case, the Milnor number can be interpreted as the intersection index of $g$ and $f$ at the point $0=(0,0)$, that is, \[\mu(\F_{d})=I_{0}(f,g),\] see \cite{fulton2008algebraic} by the properties of the intersection index.

If we write the components of $\chi_{d}$ as
\[
    f = f_{m}+f_{m+1}+\cdots, \mbox{ and }  ~ g = g_{s} +g_{s+1}+\cdots. 
\]
then, the {\bf algebraic multiplicity} of $\F_{d}$ at $p$ is defined as \[m_{p}(\F_{d}):=\min\{m,s\}.\]

\begin{defn}
    An irreducible plane curve defined by a polynomial $F(x,y,z)$ is an {\bf algebraic solution} for $\F_{d}$ or {\bf invariant} under $\F_{d}$ if and only if there exists a polynomial $H(x,y,z)$ such that 
    \[P(x,y,z)\frac{\partial F(x,y,z)}{\partial x}+Q(x,y,z)\frac{\partial F(x,y,z)}{\partial y}+R(x,y,z)\frac{\partial F(x,y,z)}{\partial z}=F(x,y,z)H(x,y,z).\]
\end{defn}

An important result concerning this definition was established by Jouanolou, Lins Neto and Soares.

\begin{thm}[\cite{neto1996algebraic}]
    For $d\geq 2$, the subset $\{\F_{d}\in\F(d;2) | \F_{d} \mbox{ has no algberaic solutions}\}$ is not-empty, dense in $\F(d;2)$, and contains an open dense subset.
\end{thm}

However, finding explicit examples of this type of objects is not an easy task.

\subsection{Action by coordinate change}
Since $\F(d;2)$ is a projective variety of dimension $d^{2}+4d+2$, we can consider the action of $SL_{3}(\C)$ on $\F(d;2)$ given by coordinate changes.

\[\begin{array}{rcl}
  SL_{3}(\C)\times \F(d;2)&\rightarrow & \F(d;2) \\
    (g,X)&\mapsto & g\cdot X=DgX(g^{-1}).
  \end{array}\]

  It is well known that every $1-$parameter subgroup of $SL_{3}(\C)$ can be written as 
  \[g\lambda(t)g^{-1}=g\begin{pmatrix} t^{k_{1}}&0&0\\ 0&t^{k_{2}}&0\\0&0& t^{k_{3}}\end{pmatrix}g^{-1}, \] for some $g\in SL_{3}(\C)$, where $k_{1}\geq k_{2}\geq k_{3}$ and $k_{1}+k_{2}+k_{3}=0$. 

  We denote by $\lambda_{(k_{1},k_{2})}$ the diagonal $1$-PS \[\begin{pmatrix} t^{k_{1}}&0&0\\ 0&t^{k_{2}}&0\\0&0& t^{k_{3}}\end{pmatrix},\] and we will assume that the integers $k_{1},k_{2},k_{3}$ are relatively primes. 

The action by coordinate change interacts well with the invariants of our interest.

\begin{prop}(\cite{Ramon})
    Let $g\in SL_{3}(\C)$ and $X$ be a vector field associated with the foliation $\F_{d}$. Then 
    \begin{itemize}
        \item $p\in Sing(\F_{d})$ if and only if $g(p)\in Sing(g\cdot \F_{d})$.
        \item $\mu_{p}(\F_{d})=\mu_{g(p)}(g\cdot \F_{d})$.
        \item $m_{p}(\F_{d})=m_{g(p)}(g\cdot \F_{d})$.
    \end{itemize}
\end{prop}

where $g\cdot \F_{d}=g\cdot X$ for convenience. 

Let $X$ be a vector field defining a foliation as in  (\ref{Foliacion}). Consider a component of $X$ of the form
\[X_{0}=\begin{pmatrix}
    x^{\alpha_{1}}y^{\beta_{1}}z^{\gamma_{1}}\\
    x^{\alpha_{2}}y^{\beta_{2}}z^{\gamma_{2}}\\
    x^{\alpha_{3}}y^{\beta_{3}}z^{\gamma_{3}}\\
\end{pmatrix},\] and let   $\lambda_{(k_{1},k_{2})}$  be a $1$-PS of $SL_{3}(\C)$. Then the action of $\lambda_{(k_{1},k_{2})}$ on $X_{0}$ is given by 
\[\lambda_{(k_{1},k_{2})}\cdot X_{0}:=\lambda_{(k_{1},k_{2})}X_{0}\left(\lambda_{(k_{1},k_{2})}^{-1}\begin{pmatrix} x\\y\\z\end{pmatrix}\right) =\begin{pmatrix}
    t^{-k_{1}(\alpha_{1}-1)-k_{2}\beta_{1}-k_{3}\gamma_{1}}x^{\alpha_{1}}y^{\beta_{1}}z^{\gamma_{1}}\\
    t^{-k_{1}\alpha_{2}-k_{2}(\beta_{2}-1)-k_{3}\gamma_{2}}x^{\alpha_{2}}y^{\beta_{2}}z^{\gamma_{2}}\\
    t^{-k_{1}\alpha_{3}-k_{2}\beta_{3}-k_{3}(\gamma_{3}-1)}x^{\alpha_{3}}y^{\beta_{3}}z^{\gamma_{3}}\\
\end{pmatrix}.\]

Now, we consider the monomial vector fields

\begin{align}\label{FoliacionesMonomios}
    X_{i_{0},j_{0}}^{0,d}&= x^{d-j_{0}}y^{j_{0}-i_{0}}z^{i_{0}}\frac{\partial}{\partial x},
    \end{align}
\begin{align*}
    X_{i_{1},j_{1}}^{1,d}&= x^{d-j_{1}}y^{j_{1}-i_{1}}z^{i_{1}}\frac{\partial}{\partial y},\\
    X_{i_{2},j_{2}}^{2,d}&= x^{d-j_{2}}y^{j_{2}-i_{2}}z^{i_{2}}\frac{\partial}{\partial z}.
\end{align*}

with $j_{l}\geq i_{l}$ for all $l=0,1,2$ and $j_{l},i_{l}\in [d]^{0}:=\{0,1,2,\ldots,d\}$. Thus, the weights of components $X_{i_{l},j_{l}}^{l,d}$ with respect to $\lambda_{(k_{1},k_{2})}$ are given by:

\begin{align}\label{CondicionPesos}
C^{l,d}_{i_{l},j_{l}}:=-k_{1}(d-j_{l})-k_{2}(j_{l}-i_{l})-k_{3}i_{l}+k_{l+1} \mbox{ for } l=0,1,2.
\end{align}

\begin{rem}\label{K3K1}
Let us observe that, by the conditions $k_{1}+k_{2}+k_{3}=0$ and $k_{1}\geq k_{2}\geq k_{3}$, if there is $\lambda_{(k_{1},k_{2})}$ a $1$-SP such that $C^{l,d}_{i_{l},j_{l}}>0$ then $-2\leq \frac{k_{3}}{k_{1}} \leq -\frac{1}{2}$.
\end{rem}

Consider the set of generators of $\F(d;2)$ given by 
\[\left\{ X_{i_{0},j_{0}}^{0,d}, X_{i_{1},j_{1}}^{1,d},X_{i_{2},j_{2}}^{2,d}~| ~ j_{l}\geq i_{l}~\forall ~l=0,1,2, \mbox{ and } j_{l},i_{l}\in [d]^{0} \right\}.\] 

We can then  consider  a decomposition the space of scalar subrepresentations of $\C^{d^2+4d+3}$ as \[\C^{d^{2}+4d+3}=\bigoplus_r \alpha_r V_{\alpha_r},~r=1,\ldots,d^2+4d+3, \] where $\alpha_r: T\rightarrow \C$.

Moreover, to this representation we can associate a set of points in $\R^2$ given by 
\[
\left\{-L_{1}(d-j_{l}-i_{l})-L_{2}(j_{l}-2i_{l})+L_{l+1} | l=0,1,2. \right\}\subset \R^2.
\]
where $L_1= (1,0), L_2=(-\frac{1}{2},\frac{\sqrt{3}}{2})$ and $L_3=(-\frac{1}{2},-\frac{\sqrt{3}}{2})$.

We refer to this set of points in $\R^2$ as the {\bf weight diagram } of the representation. Some examples are illustrated in Figure (\ref{fig:Diagramas}).

\begin{figure}[ht]
\centering
\begin{subfigure}[b|]{0.45\linewidth}
\includegraphics[width=0.87\textwidth]{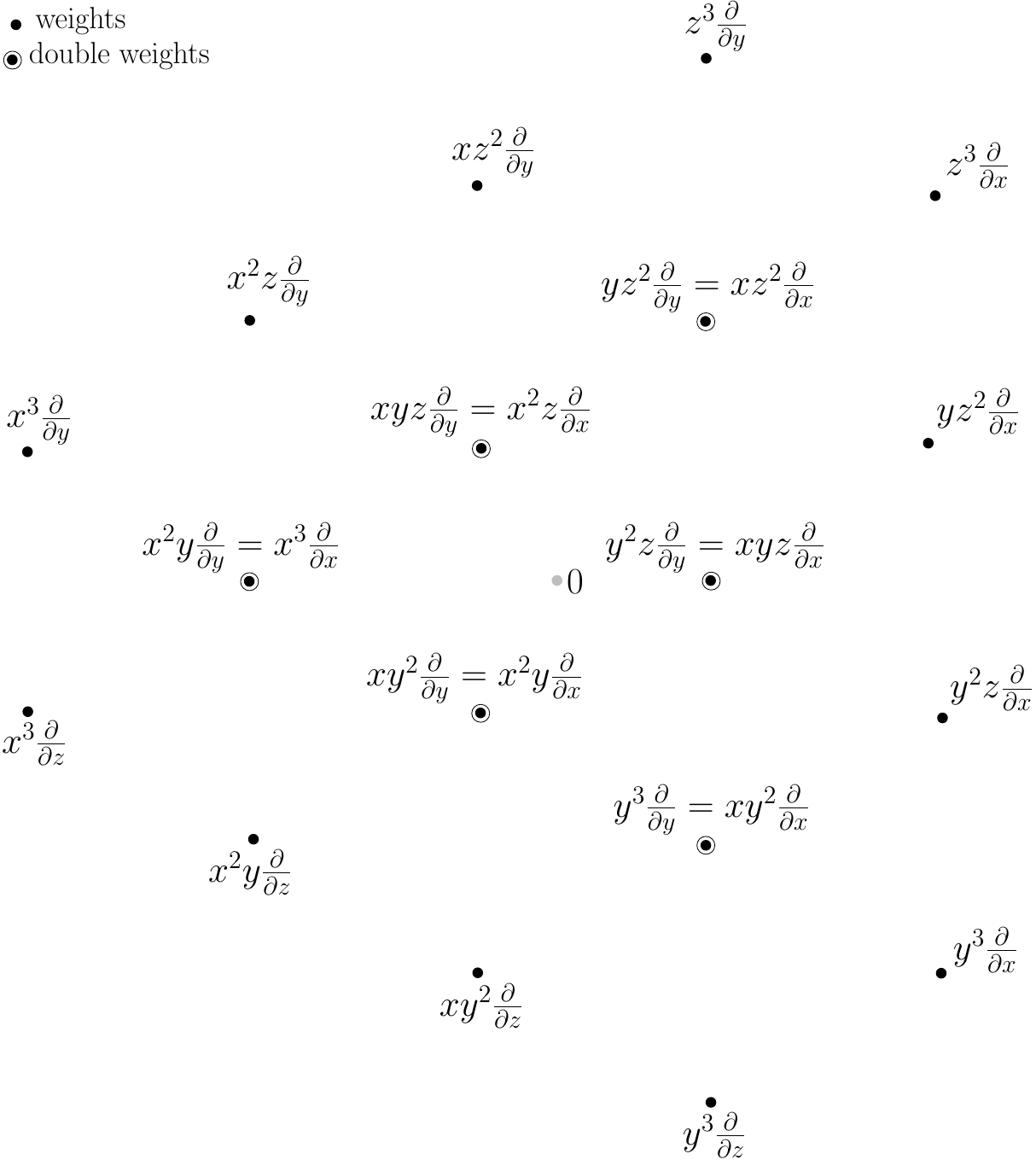}
\caption{Weight diagram of $\F(3,2)$.}
\label{fig:grado3}
\end{subfigure}
\begin{subfigure}[b]{0.45\linewidth}
\includegraphics[width=0.87\linewidth]{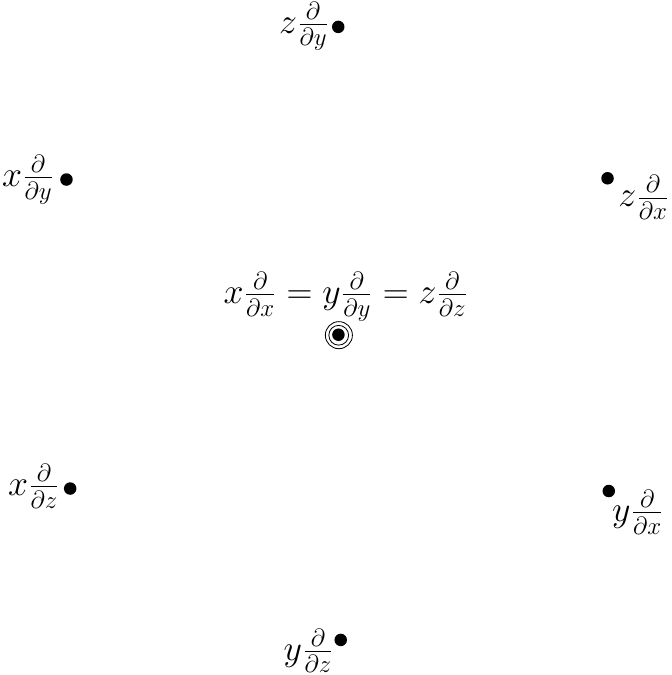}
\caption{Weight diagram of $\F(1,2)$.}  
\label{fig:grado1}
\end{subfigure}
\caption{Weight diagram of the representation.}
\label{fig:Diagramas}
\end{figure}

It is well known that the numerical criterion of the Theorem (\ref{definicionEstabilidad}) can  be interpreted as a criterion on convex sets in the weight diagram. More precisely, let  $v=[v_0:\dots:v_n]\in V\subset \P^n$, then:
\begin{itemize}
    \item  The point $v$ is semistable if $0\in  Conv\{\alpha_i\mid v_i\neq 0\}$.
    \item   The element $v$ is stable if $0$ lies in the interior of $Conv\{\alpha_i\mid v_i\neq 0\}$.
    \item The point $v$ is unstable if $0\notin Conv\{\alpha_i\mid v_i\neq 0\}$. 
\end{itemize}

\section{Results}\label{results}
In this section, we present some results concerning Mumford stability for foliations of degree $d$. We focus on the local invariants of their singularities.

\begin{defn}
    A point in the weights diagram is called {\bf double} if there are two weights of the representation  corresponding to the same point. 
\end{defn}
For example, in the Figure \ref{fig:Diagramas}, the origin is a double point in the weights diagram of $\F(1,2)$ (Figure \ref{fig:grado1}) but it is not a double point for $\F(3,2)$ (Figure \ref{fig:grado3}).

\begin{thm}
Let $m\geq 0$ be an integer number. The origin $(0,0)$ is a double weight in the weight diagram of the representation on $\F(d;2)$ if and only if $d=3m+1$.
\end{thm}

\begin{proof}
If $d=3m+1$, then $x^{m+1}y^{m}z^{m}$ and $x^{m}y^{m+1}z^{m}$ are monomials of degree $d$. 
    For $i_{0}=m$, $j_{0}=2m$, $i_{1}=m$ and $j_{1}=2m+1$ in the monomial vector fields, we have 
\begin{align*}
 C^{0,3m+1}_{i_{0},j_{0}}& =-k_{1}(m+1-1)-k_{2}m-k_{3}m= -m(k_{1}+k_{2}+k_{3})=0,\\
C^{1,3m+1}_{i_{1},j_{1}}& =-k_{1}m-k_{2}(m+1-1)-k_{3}m= -m(k_{1}+k_{2}+k_{3})=0.\\
\end{align*}

On the other hand, by the equivalent relation of foliations, we can assume $i_{2}=0$. Then, if \[C_{0,j_{2}}^{2,d}=-k_{1}(d-j_{2}-1)-k_{2}(j_{2}+1)=0  \mbox{ and } (k_{1},k_{2})=1, \] we have the following cases:
\begin{itemize}
    \item[*)] Case $k_{1}>k_{2}$: Then $j_{2}=-1$, which is a contradiction.
\end{itemize}
Therefore, if the origin is a double weight, it must occur in the monomial vector fields $X_{i_{0},j_{0}}^{0,d}$ and $X_{i_{1},j_{1}}^{1,d}$.
\begin{itemize}  
    \item[*)] Case $k_{1}>k_{2}=k_{3}$. If  
    \begin{align*}
        C^{0,d}_{i_{0},j_{0}}=& -k_{1}(d-j_{0}-1)-k_{2}(j_{0})=0, \mbox{ and }\\
        C^{1,d}_{i_{1},j_{1}}=&- k_{1}(d-j_{1})-k_{2}(j_{1}-1)=0,
    \end{align*}
    then $j_{0}=0$, $j_{1}=1$, and $d=1=3*0+1$. In this case, the monomial vector fields are \[X_{0,0}^{0,1}=x\frac{\partial}{\partial x} \mbox{ and } X_{0,1}^{1,1}=y\frac{\partial}{\partial y},\] see Figure (\ref{fig:grado1}).
    \item[*)] Case $k_{1}> k_{2} > k_{3}$. If 
     \begin{align*}
        C^{0,d}_{i_{0},j_{0}}=& -k_{1}(d-j_{0}-1-j_{0})-k_{2}(j_{0}-2i_{0})=0, \mbox{ and }\\
        C^{1,d}_{i_{1},j_{1}}=&- k_{1}(d-j_{1}-1)-k_{2}(j_{1}-i_{1}-2)=0.
    \end{align*}
    then $j_{0}=2i_{0}$, $j_{1}=2i_{1}+1$, and $d=3i_{0}+1=3i_{1}+1$.  Moreover, the monomial vector fields  are \[X^{0,3i_{0}+1}_{i_{0},2i_{0}}=x^{i_{0}+1}y^{i_{0}}z^{i_{0}}\frac{\partial }{\partial x } \mbox{ and } X^{1,3i_{0}+1}_{i_{0},2i_{0}+1}=x^{i_{0}}y^{i_{0}+1}z^{i_{0}}\frac{\partial}{\partial y}.\]
\end{itemize}
\end{proof}
\begin{cor}
    If $d=3m+1$ for some $m\geq 1$. If the  vector field $X$ contains a  component \[X_{m,2m}^{0,d}=x^{m+1}y^{m}z^{m}\frac{\partial}{\partial x} \mbox{ or } X_{m,2m+1}^{1,d}=x^{m}y^{m+1}z^{m}\frac{\partial}{\partial y},\]  then $X$ is a semistable point. Moreover, $X$ has a singular point $p$ satisfying \[m_{p}(X)\leq m+1.\]
\end{cor}

We denote by $P(y,z)_{d}=\sum_{s=0}^{d} a_{s}y^{s}z^{d-s}$ and $Q(y,z)_{d}=\sum_{s=0}^{d} b_{s}y^{s}z^{d-s}$ the homogeneous  polynomials of degree $d$, where  $a_{i},b_{i}\in\C$.

\begin{thm}\label{ddd}
The foliation  
 $X_{ddd}=\begin{pmatrix}
    P(y,z)_{d}\\
    Q(y,z)_{d}\\
    c_{d}y^{d}\\
\end{pmatrix}\in \F(d;2)$ is unstable.
\end{thm}

\begin{proof}
Since the exponents of the components $P(y,z)_{d}$ and $Q(y,z)_{d}$ are of the form $(0,\beta_{i},d-\beta_{i})$ for $i=1,2$ and $0\leq \beta_{i}\leq d$, we obtain the following equations:
\begin{align*}
   1) &  k_{1}-k_{2}\beta_{1}-k_{3}(d-\beta_{1}) =k_{1}-k_{2}\beta_{1}+(k_{1}+k_{2})(d-\beta_{1})=(1+d-\beta_{1})k_{1}+(d-2\beta_{1})k_{2}.\\
   2) &  -k_{2}(\beta_{2}-1)-k_{3}(d-\beta_{2}) =-k_{2}(\beta_{2}-1)+(k_{1}+k_{2})(d-\beta_{2}) =k_{2}(1+d-2\beta_{2})+(d-\beta_{2})k_{1}.\\
 3) & -k_{2}d+(k_{1}-k_{2})= k_{2}(-d-1)+k_{1}.
\end{align*}

We now consider the following set of points on $\R^{2}$:
\[\{(1+d-\beta_{1},d-2\beta_{1})~|~0\leq \beta_{1}\leq  d\} \cup \{(d-\beta_{2}, 1+d-2\beta_{2})~|~0\leq \beta_{2} \leq d\} \cup \{(1,-d-1)\}.\]

After a change of basis, we obtain: $W=\bigcup_{h=0}^{2} W_{h}$,  where 
\begin{align*}
    W_{0}=& \{(1+\frac{1}{2}d, (d-2\beta_{1})\frac{\sqrt{3}}{2})~|~ 0\leq \beta_{1}\leq d\},\\
    W_{1}=& \{(\frac{1}{2}(d-1),(1+d-2\beta_{2})\frac{\sqrt{3}}{2})~|~ 0\leq \beta_{2}\leq d\}, \\
    W_{2}=&\{(\frac{1}{2}(d-1),(-d-1)\frac{\sqrt{3}}{2})\}.
    \end{align*}

We can see that the points in the set $W_{0}$ 
lie on a straight line, and so do the points in  $W_{1}\cup W_{2}$. 

Since $d\geq 2$, we have  $1+\frac{1}{2}d>\frac{1}{2}d-\frac{1}{2}\geq \frac{1}{2}$. Therefore, the convex hull of the set $W$ does not contain the origin $(0,0)$.

 Hence, the foliation $X_{ddd}$ is unstable.
\end{proof}

\begin{cor}
    The foliation $X_{ddd}$ has no degenerated singularities, that is, for every  $p\in Sing(X_{ddd})$, the Milnor number satisfies $\mu_{p}(X_{ddd})> 1$.
\end{cor}

For particular cases of $X_{ddd}$ we have the following  result with respect to their singularities. 

 \begin{thm}\label{dd0}
 If $X_{dd0}=\begin{pmatrix}
    P(y,z)_{d}\\
    Q(y,z)_{d}\\
    0  \\
\end{pmatrix}\in \F(d;2)$  whit isolated singularities, then we have one of the following conditions:
\begin{enumerate}
    \item the singular set of $X_{dd0}$ is a unique singular point or,
    \item the singular set of $X_{dd0}$  consists of two singular points which one it is of Milnor number $1$ and the other it has  Milnor number $d^{2}+d$. 
\end{enumerate}
\end{thm}

\begin{proof}
    Since the point $p=[1:0:0]\in Sing(X_{dd0})$, consider the local representation at the point $p$ given by the polynomials: $f=Q(y,z)-yP(y,z)$ and $g=-zP(y,z)$. The intersection index of $f$ and $g$ in the point $0=(0,0)$ is \[I_{0}(f,g)=I_{0}(f,z)+I_{0}(Q,P).\]

    If $Sing(X_{dd0})$ has isolated singularities, then $P$ and $Q$ have no common factors, which implies that $I_{0}(P,Q)=d^{2}$. Moreover, 

  \begin{equation}
  I_{0}(f,z)=\left\{
   \begin{array}{lll}
     1+d & \mbox{ if }  z|Q,\\
     d & \mbox{ if } z\nmid Q\\
    \end{array}
  \right.
\end{equation}

Thus, 

\begin{equation}
  I_{0}(f,g)=\left\{
   \begin{array}{lll}
     1+d+d^2 & \mbox{ if }  z|Q,\\
     d+d^2 & \mbox{ if } z\nmid Q.\\
    \end{array}
  \right.
\end{equation}

By Jouanolou (\cite{J06}), $X_{dd0}$ has a unique singular point or has two points.
\end{proof}

\begin{thm}
    Let $X_{d}=\begin{pmatrix}
    P(y,z)_{d}\\
    Q(y,z)_{d}\\
    y^{d}\\
\end{pmatrix}\in \F(d;2)$ be a foliation. If $Q(y,z)_{d}=y^{k}Q_{d-k}$ for $1\leq k\leq d$, $y\nmid Q_{d-k}$ and $(P,Q)=1$ then $X_{d}$ has a singular point with Milnor number $d^{2}+k$ and algebraic multiplicity $d$.
\end{thm}

\begin{proof}
    Since $p=[1:0:0]\in Sing(X_{d})$, consider the local representation of $X_{d}$ at $p$ given by $f=Q(y,z)_{d}-yP(y,z)_{d}$ and $g=y^{d}-zP(y,z)_{d}$. Observe that $m_{p}(X_{d}))=d$. Thus, the intersection index of $f$ and $g$ at the point $(0,0)$ is 
    \begin{equation}\label{k}
  I_{0}(f,g)=\left\{
   \begin{array}{lll}
     d^{2} & \mbox{ if }  k=0,\\
     d^{2}+k & \mbox{otherwise}.
    \end{array}
  \right.
\end{equation}

Indeed, we assume that $Q_{d}=y^{k}Q_{d-k}$ and $y\nmid Q_{d-k}$ for some $0\leq k \leq d$. 
Define the following  polynomials \[A:=-y^{d-k}f+Q_{d-k}g=-y^{d-k}y^{k}Q_{d-k}+y^{d-k}yP+Q_{d-k}y^{d}-zQ_{d-k}P=P(y^{d+1-k}-zQ_{d-k}).\] and \[B:=y^{d+1-k}-zQ_{d-k}.\]

Let us observe that $A=PB$. Moreover, $I_{0}(f,g)=I_{0}(A,f)-I_{0}(Q_{d-k},f)$.
Let's calculate each intersection index.

Since $I_{0}(A,f)=I_{0}(P,f)+I_{0}(B,f)$. We can see the following:
\begin{enumerate}
    \item $I_{0}(P,f)=I_{0}(P,Q)=d^{2}$, because $(P,Q)=1$.
    \item $B$ is a homogeneous polynomial of degree $d+1-k$ that $y, Q_{d-k}\nmid B$ then \[I_{0}(B,f)=(d+1-k)(d)=d^{2}+d-kd.\]
    \end{enumerate}\label{A2}

On the other hand, 
\begin{enumerate}
    \item[(3)] $I_{0}(Q_{d-k},f)=I_{0}(Q_{d-k},y^{k}Q_{d-k}-yP)=I_{0}(Q_{d-k},y)+I_{0}(Q_{d-k},P)=(d-k)+(d-k)d=d^{2}+d-k-dk$ because $(P,Q)=1$ and $y\nmid Q_{d-k}$.
\end{enumerate}\label{A1}

By the equations $(1)-(3)$ we have:
\begin{align*}
  I_{0}(f,g)=&I_{0}(P,f)+I_{0}(B,f)-I_{0}(Q_{d-k},f)\\
  =&d^{2}+(d^{2}+d-kd)-(d^{2}+d-k-dk)\\
  =&d^{2}+k.  
\end{align*}

\end{proof}

\begin{thm}\label{multiplicidad-Fija}
Let $X=\begin{pmatrix}
    P(y,z)_{d}\\
    Q(x,y,z)_{d}\\
    0\\
\end{pmatrix}\in \F(d;2)$ be a foliation. $X$ has a singular point of multiplicity $m$ and Milnor number between $m(d+1)$ and $d^{2}+d+1$. 
\end{thm}

\begin{proof}
Let $Q(x,y,z)=Q^{d}+xQ^{d-1}+\ldots+x^{d-1}Q^{1}+x^{d}Q^{0}$ where $Q^{i}$ are homogeneous polynomials of degree $i$ in the polynomial ring $\C[y,z]$. We denote by $Q=Q(1,y,z)$. If we assume that the point $p=[1:0:0]$ is a singular point of $X$ then $Q^{0}\equiv 0$. Since the local representation of $X$ around of point $p$ is given by the polynomials $f=Q-yP; g=-zP$. Then \[I_{0}(f,g)=I_{0}(f,z)+I_{0}(P,Q).\]

If we write $f$ as its homogeneous parts, $f=f^{1}+f^{2}+\ldots + f^{d+1}$. Consider $m=min\{i\in [d]~|~f^{i}\neq 0\}$. Since $f^{i}=Q^{i}$ for $1\leq i\leq d$ and $f^{d+1}=-yP$ then $Q=Q^{d}+\ldots +Q^{m}$. Therefore, the following holds.
  \begin{equation}\label{ml}
  I_{0}(f,z)=\left\{
   \begin{array}{lll}
        m   & \mbox{ if }  z\nmid Q^{m}\\
        l+k & \mbox{ if }  z |(f^{m}+f^{m+1}+\ldots +f^{l}) \mbox{ and } z\nmid f^{l+k};\\
            & \mbox{ where } m\leq l < l+k \leq d+1, \mbox{ and } f^{l+1}=\ldots=f^{l+k-1}=0.\\
    \infty  & \mbox{ if } l=d+1.  
    \end{array}
  \right.
\end{equation}

If $m=d$ and $l+k=d+1$ then $I_{0}(f,z)=d+1$. Since $(P,Q)=1$ then $I_{0}(P,Q)=d^{2}$, thus $I_{0}(f,g)=d^{2}+d+1$. 

Since $P$ is a homogeneous polynomial of degree $d$ in the polynomial ring $\C[y,z]$, then $P$ is a product of $d$ lines $p_{i}\in \C[y,z]_{1}$;  \[P=\prod_{i=1}^{d} p_{i}.\] Thus, $I_{0}(P,Q)=\sum_{i=1}^{d}I_{0}(p_{i},Q)$. Observe the following.
  \begin{equation}\label{piQ}
  I_{0}(p_{i},Q)=\left\{
   \begin{array}{lll}
        m   & \mbox{ if }  p_{i}\nmid Q^{m}\\
      s_{i}+r_{i}   & \mbox{ if }  p_{i}|(Q^{m}+Q^{m+1}+\ldots + Q^{s_{i}}) \mbox{ and } p_{i}\nmid Q^{s_{i}+r_{i}}\neq 0;\\
            & \mbox{ where } m\leq s_{i} < s_{i}+r_{i} < d, \mbox{ and } Q^{s_{i}+1}=\ldots=Q^{s_{i}+r_{i}-1}=0.\\
    \infty  & \mbox{ if } s_{i}=d.  
    \end{array}
  \right.
\end{equation}

If we assume that $s_{i}\neq d$ for all $i\in [d]$ then, 
\[md\leq I_{0}(P,Q)\leq d^{2}.\]

With the equations (\ref{ml}) and (\ref{piQ}), the Milnor number at the point $p$ is $m(d+1)\leq \mu_{p}(X)\leq d^{2}+d+1$. Since  $P$ is homogeneous polynomial of degree $d$, the multiplicity of $p$ in $X$ is give by the $Q$ polynomial.

\end{proof}

\begin{cor}\label{M1}
Let $X=\begin{pmatrix} P(y,z)_{d}\\ Q(x,y,z)_{d}\\0 \end{pmatrix}\in \F_{d}$ be a foliations of degree $d$.  If $p=[1:0:0]\in Sing(X)$ with multiplicity $m_{p}(X)=1$ then the Milnor number of $p$ in $X$ is between $d+1$ and $d^{2}+d+1$.  
\end{cor}

\begin{cor}
For $d=4$. If $X$ is a foliation as in the corollary (\ref{M1}), then $X$ has at least two singularities. 
\end{cor}

\begin{proof}
    Let $X=\begin{pmatrix}
        P(y,z)_{4}=\prod_{i=1}^{4}p_{i}\\ Q(x,y,z)_{4}=\sum_{i=1}^{4}x^{i}Q_{4-i}\\0
    \end{pmatrix}\in\F(4;2)$ a foliation where $p_{i}\in\C[y,z]_{1}$ and $Q_{4-i}\in \C[y,z]_{4-i}$ for all $i=1,\ldots 4$. We assume that $\{p=[1:0:0]\}=Sing(X)$, then in local coordinates, we consider $f=\sum_{i=1}^{4}Q_{i}-y*P(y,z)_{4}$ and $g= -z*P(y,z)_{4}$ with $f(0,0)=g(0,0)=0$.  Then $I_{0}(f,g)=21$ and this occurs if and only if $I_{0}(f,z)=5$ and $I_{0}(P,Q)=16$ in inequalities (\ref{ml}) and (\ref{piQ}).
   However, if $I_{0}(f,z)=5$, this implies that $z$ divides to $Q_{1}+Q_{2}+Q_{3}+Q_{4}=Q$. Since  $I_{0}(P,Q)=16$, it follows that $p_{i}=z$ for all $i=1,\ldots,4$.  Therefore, $z$ is a common factor of the foliation $X$, and consequently $X$ does not have isolated singularities. 
\end{proof}

A natural question arising from the foliations in the Lemma \ref{multiplicidad-Fija} concerns their stability. When the algebraic multiplicity is $d$, we have the following result.

\begin{lem}
    Let $X$ be a foliation of degree $d\geq 2$ with a singular point of multiplicity $d$. Then $X$ is a unstable foliation. 
\end{lem}

\begin{proof} Let $X=P(x,y,z)\frac{\partial }{\partial x}+Q(x,y,z)\frac{\partial }{\partial y}+R(x,y,z)\frac{\partial }{\partial z}$ be a foliation of degree $d$, that is, $P,Q,R\in \C[x,y,z]_{d}$. If we assume that $p=[1:0:0]\in Sing(X)$ such that $m_{p}(X)=d$. By definition of multiplicity,
\[X=\begin{pmatrix}
    xP(y,z)_{d-1}+P(y,z)_{d}\\
    Q(y,z)_{d}\\
    R(y,z)_{d}
\end{pmatrix}\]
 for some homogeneous polynomials $P(y,z)_{d},P(y,z)_{d-1}, Q(y,z)_{d}, R(y,z)_{d}\in \C[y,z]$ or some them equal to zero polynomial.  

 By the notation of equations  (\ref{FoliacionesMonomios}), we assume that
 \begin{align*}
     P(y,z)_{d-1}&=\sum_{i_{0}=0}^{d-1} a_{i_{0},d-1}y^{d-1-i_{0}}z^{i_{0}},\\
     P(y,z)_{d}&=\sum_{i_{0}=0}^{d} a_{i_{0},d}y^{d-i_{0}}z^{i_{0}},\\
     Q(y,z)_{d}&=\sum_{i_{1}=0}^{d} b_{i_{1},d}y^{d-i_{1}}z^{i_{1}},\\
     R(y,z)_{d}&= c_{0,d}y^{d},
\end{align*}
 with $a_{i_{0},*},b_{i_{1},d},c_{0,d}\in \C$. We consider the equation (\ref{CondicionPesos}) 

\[C^{l,d}_{i_{l},j_{l}}=k_{1}(d-i_{l})+k_{3}(d-2i_{l})+k_{l+1}= k_{1}A_{l}+k_{3}B_{l}\] for $l=0,1,2$.
Then 
 \begin{align*}
 A_{0}=&d+1-i_{0},~ B_{0}=d-2i_{0},\\
 A_{1}=&d-i_{1}-1,~ B_{1}=d-1-2i_{1},\\
 A_{2}=&d-i_{2},~ B_{2}=d-2i_{2}+1.
 \end{align*}
with $0\leq i_{l}\leq d$. These weights correspond to monomial foliations  $X^{l,d}_{i_{l},d}$. 
 To calculated that $C_{i_{l},j_{l}}^{l,d}>0$  we need to consider two cases.

Before reviewing each case, we observe the following inequalities:

\begin{itemize}
    \item   \[A_{0}=d+1-i_{0} >0 \mbox{ for all } 0\leq i_{0}\leq d.\]
    \item  \begin{equation*}
   A_{1}=d-1-i_{1} \mbox{ then } \left\{
   \begin{array}{lll}
        A_{1} < 0   &  \mbox{ if }  d=i_{1}, \\
        A_{1} = 0   &  \mbox{ if } d-1=i_{1}, \\
        A_{1} >0    & \mbox{ if } d-1 > i_{1}.
    \end{array}
  \right.
\end{equation*}
\item \begin{equation*}
  A_{2}=d-i_{2} \mbox{ then } \left\{
   \begin{array}{lll}
        A_{2}= 0   &  \mbox{ if } d=i_{2}, \\
        A_{2}>0    & \mbox{ if }  i_{2}\neq d.
    \end{array}
  \right.
\end{equation*}
\end{itemize}

\subsubsection{Case $l=0$}
\[C_{i_{0},d}^{0,d}=k_{1}A_{0}+k_{3}B_{0}> 0.\]
Since $A_{0}> 0$, we analyze the following subcases;

\begin{enumerate}
    \item $B_{0}<0$, that is, $\frac{d}{2} < i_{0}$. The condition $C_{i_0,d}^{0,d}>0$ is hold if and only if $\frac{k_{3}}{k_{1}}< \frac{A_{0}}{B_{0}} \leq -\frac{1}{2}$. This is true because $\frac{d}{2}< \frac{3d+2}{4}\leq i_{0}$.
    \item $B_{0}>0$, that is, $\frac{d}{2} > i_{0}$. The condition $C_{i_{0},d}^{0,d}>0$ is hold if and only if $\frac{k_{3}}{k_{1}} > -\frac{A_{0}}{B_{0}} \geq -2$. This is true if and only if  $i_{0}\leq \frac{d-1}{3}$, that means when $i_{0}<\frac{d}{2}$.
    \item $B_{0}=0$, if $i_{0}=\frac{d}{2}$, which happens when $d$ is even.  In this case the condition $C_{i_{0},d}^{0,d}>0$ is reduced to $k_{1}A_{0}>0$ which always happens.
    \end{enumerate}

\subsubsection{Case $l=1$}
\[C_{i_{1},d}^{1,d}=k_{1}A_{1}+k_{3}B_{1}> 0.\]
Since \begin{equation*}
   A_{1}=d-1-i_{1} \mbox{ then } \left\{
   \begin{array}{lll}
        A_{1} < 0   &  \mbox{ if }  d=i_{1}, \\
        A_{1} = 0   &  \mbox{ if } d-1=i_{1}, \\
        A_{1} >0    & \mbox{ if }  i_{1}\in [0,d-1)\cap \Z,
    \end{array}
  \right.
\end{equation*}
and 
\begin{equation*}
   B_{1}=d-1-2i_{1} \mbox{ then } \left\{
   \begin{array}{lll}
       B_{1} < 0   &  \mbox{ if }  i_{1} \in (\frac{d-1}{2},d]\cap\Z, \\
        B_{1}= 0   &  \mbox{ if } \frac{d-1}{2}=i_{1} \mbox{ when } d \mbox{ is  impar }, \\
        B_{1}>0    & \mbox{ if }  i_{1}\in[0,\frac{d-1}{2})\cap\Z.
    \end{array}
  \right.
\end{equation*}

In these cases we have the following  subcases:

\renewcommand{\arraystretch}{1.1} 
 \begin{figure}[!h]
\begin{tabular}{|c|c|c|c l|}\hline
$A_{1}$ &  $B_{1}$ & values of $i_{1}$                & $C_{i_{1},d}^{1,d}>0$  & \\\hline
$>0$ &   $<0$    & $[0,\frac{d-1}{2})\cap\Z$         & $\frac{k_{3}}{k_{1}}< \frac{A_{1}}{-B_{1}}\leq -\frac{1}{2}$& Happens if $d\geq 1$.\\
$>0$ &   $=0$    &  $\frac{d-1}{2}$                  & $k_{1}>0$& It is true. \\
$>0$ & $>0$ & $(\frac{d-1}{2}, d-1)\cap\Z$  & $\frac{k_{3}}{k_{1}}> -\frac{A_{1}}{B_{1}}\geq -2$ &  It is true if $d>1$.\\
$=0$     &   $>0$    & $d-1$                             & $k_{3}<0$ & It is true.\\
$<0$     &   $<0$    & $d$ & $\frac{k_{3}}{k_{1}} > -\frac{1}{d+1}\geq -2$ & Is is true. \\ \hline
 \end{tabular}
\end{figure}

\subsubsection{Case $l=2$}

By equivalent of foliations, $i_{2}=0$, then $A_{2}=d$, and $B_{2}=d+1$, thus $C_{0,d}^{2,d}=dk_{1}+(d+1)k_{3}>0$ if only if, $\frac{k_{3}}{k_{1}}>-\frac{d}{d+1}\geq -2$ if and only if $0\leq d+2$.

We consider the equation given by: 

\[C_{d-1,0}=k_{1}(d-1-i_{0})+k_{3}(d-1-2i_{0})\]   whose weights correspond to the monomials of $P(y,z)_{d-1}$ and its analysis is restricted to the case $A_{1},B_{1}$ above. 
\end{proof}
The converse is not true: there exist unstable foliations with multiplicity less than $d$, see \cite{CR16,castorena2024git} for examples.

\begin{thm}
    Let $\F$  be a foliation of degree $d$ with a finite abelian group of automorphisms and a unique singular point. Let $m$ the multiplicity of the singularity, then if one of the following conditions is satisfied then $\F$ is stable:
    \begin{itemize}
        \item $m\leq \frac{d-1}{3}$ or 
        \item $\frac{d-1}{3}< m \leq \frac{d}{2}$ and $\F$ have not invariant lines.
    \end{itemize} 
\end{thm}

\begin{proof}
    If we assume that $\F$ is an unstable foliation, by the \cite[Theorem 1.1]{A10U} we have that $m > \frac{d-1}{3}$ and have a invariant line or $m> \frac{d}{2}$. Then $\F$ is semistable, but since $\F$ has a finite group  of automorphism then $\F$ is stable.
\end{proof}

\begin{thm}
    Let $\F$ be a stable foliation of degree $d$ with isolated singularities. Then $\F$ has a singular point with multiplicity less than $\frac{2d+1}{3}$.
\end{thm}

\begin{proof}
    Suppose that $[1:0:0]$ is a singular point for the foliation
    \[\F: P(y,z)\frac{\partial}{\partial x}+\left(\sum\limits_{j_{1}=0}^{d}x^{d-j_{1}}Q_{j_{1}}(y,z)\right)\frac{\partial}{\partial y}+\left( \sum\limits_{j_{2}=0}^{d}x^{d-j_{2}}R_{j_{2}}(y,z)\right)\frac{\partial}{\partial z}.\] Since $\F$ is stable, then the Mumford function $\mu(\F,\lambda,)<0$ for every $1$-PS $\lambda$. 
    Consider the $1$-PS $\lambda_{(2,-1)}$. Then the weights of the action of $\lambda_{(2,-1)}$ are 
    \begin{align*}
        C^{0,d}_{i_{0},d}=&d+2,\\
        C^{1,d}_{i_{1},j_{1}}=&-2d+3j_{1}-1,\\
        C^{2,d}_{i_{2},j_{2}}=&-2d+3j_{2}-1.
    \end{align*}
    Since $d+2>0$, then there exist $j_{l}<\frac{2d+1}{3}$ for $l=1$ or $l=2$. Observe that the multiplicity of $\F$ in $[1:0:0]$ is $m=min\{j_{l}~|~j_{l}<\frac{2d+1}{3}\}$. 
\end{proof}

\begin{cor}
    Let $\F$ be a stable foliation with a unique singularity, then the multiplicity of this singularity is less that $\frac{2d+1}{3}$.
\end{cor}

\subsection*{Acknowledgements} This paper was supported with a Postdoctoral Fellowship from SECIHTI.  I would like to thank Claudia R. Alcántara for her helpful comments and suggestions. 

\end{document}